\documentclass[a4paper, 12pt]{article}
\usepackage[T1]{fontenc}
\usepackage[english]{babel}
\usepackage{latexsym}
\usepackage{amscd}
\usepackage{fancyhdr}
\usepackage{amsmath,amsfonts,amssymb,amsthm}
\usepackage[mathcal]{euscript}
\usepackage{mathrsfs}
\usepackage{geometry}
\usepackage[all]{xy}
\usepackage{textcomp}
\usepackage{epsfig}
\usepackage{graphics}
\usepackage{graphicx}
\usepackage{subfigure}
\usepackage{mathrsfs}
\usepackage[OT2,OT1]{fontenc}
\usepackage{comment}

\newcommand{\numberset}{\mathbb}
\newcommand{\N}{\numberset{N}}
\newcommand{\R}{\numberset{R}}
\newcommand{\Z}{\numberset{Z}}

\newcommand{\pp}{\ensuremath{\mathcal P}}

\theoremstyle{plain}
\newtheorem{theorem}{Theorem}

\newtheorem{corollary}{Corollary}
\newtheorem{lemma}{Lemma}
\newtheorem{conjecture}{Conjecture}
\newtheorem{problem}[conjecture]{{Problem}}

\theoremstyle{definition}
\newtheorem{definition}{Definition}

\theoremstyle{remark}

\theoremstyle{definition}

\author{Bruno Federici\thanks{Supported by an EPSRC grant EP/L505110/1.} \and Agelos Georgakopoulos\thanks{Supported by EPSRC grant EP/L002787/1.  This project has received funding from the European Research Council (ERC) under the European Union’s Horizon 2020 research and innovation programme (grant agreement No 639046). The second author would like to thank the Isaac Newton Institute for Mathematical Sciences, Cambridge, for support and hospitality during the programme `Random Geometry' where work on this paper was undertaken.}\and 
\\
 {Mathematics Institute}\\
 {University of Warwick}\\
 {CV4 7AL, UK}
}

\title{Hyperbolicity vs.\ Amenability for planar graphs}

\begin{document}

\maketitle
\begin{abstract}
The aim of this paper is to clarify the relationship between Gromov-hyperbolicity and amenability for planar maps.
\end{abstract}

\section{Introduction}
Hyperbolicity and non-amenability are important and well-studied properties for groups (where the former implies the latter unless the group is 2-ended). They are also fundamental in the emerging field of coarse geometry \cite{BenCoa}. The aim of this paper is to clarify their relationship for planar graphs that do not necessarily have many symmetries: we show that these properties become equivalent when strengthened by certain additional conditions, but not otherwise.

Let $\pp$ denote the class of plane graphs (aka.\ planar maps), with no accumulation point of vertices and with bounded vertex degrees. Let $\pp'$ denote the subclass of $\pp$ comprising the graphs with no unbounded face. We prove

\begin{theorem}
\label{coroll}
Let $G$ be a graph in $\pp'$. Then $G$ is hyperbolic and weakly non-amenable if and only if it is non-amenable and it has bounded codegree.
\end{theorem}

Here, the {\em length} of a face is the number of edges on its boundary;  a \textit{bounded} face is a face with finite length; a plane graph has \textit{bounded codegree} if  there is an upper bound on the length of bounded faces. A graph is \textit{weakly non-amenable} if satisfies an isoperimetric inequality of the form $|S|\leq f(|\partial S|)$ for all non-empty finite vertex sets $S$, where $f\colon\N\rightarrow\N$ is a monotone increasing diverging function.

Theorem~\ref{coroll} is an immediate corollary of the following somewhat finer result

\begin{theorem}
\label{The:maintheorem}
Let $G$ be a graph in \pp. Then the following hold:
\begin{enumerate}
\item if $G$ is non-amenable and has bounded codegree then it is hyperbolic;
\item if $G$ is hyperbolic and weakly non-amenable then it has bounded codegree;
\item if $G$ is hyperbolic and weakly non-amenable and in addition has no unbounded face then it non-amenable.
\end{enumerate}
\end{theorem}

We provide examples showing that none of the conditions featuring in Theorem~\ref{The:maintheorem} can be weakened, 
and that the no accumulation point condition is needed. 

We expect that Theorem~\ref{The:maintheorem} remains true in the class of 1-ended Riemannian surfaces if we replace the bounded codegree condition with the property of having bounded length of boundary components. 
\bigskip

We remark that having bounded degrees is a standard assumption, and assuming bounded codegree is not less natural when the graph is planar.
Part of the motivation behind
Theorem~\ref{The:maintheorem} comes from related recent work of the second author \cite{UKtrans,planarPB}, especially the following
\begin{theorem}[\cite{planarPB}]
\label{The:corollary1.2}
Let $G$ be an infinite, Gromov-hyperbolic, non-amenable, 1-ended, plane graph with bounded degrees and no infinite faces. Then the following five boundaries of $G$ (and the corresponding compactifications of $G$) are
canonically homeomorphic to each other: the hyperbolic boundary, the Martin boundary, the boundary of the square tiling, the Northshield circle, and the boundary $\partial_{\cong}(G)$.
\end{theorem}

In order to show the independence of the hypotheses in this, the second author provided a counterexample to a conjecture of Northshield \cite{NorCir} asking whether a plane, accumulation-free, non-amenable graph with bounded vertex degrees must be hyperbolic. That counterexample had unbounded codegree, and so the question came up of whether Northshield's conjecture would be true subject to the additional condition of bounded codegree. The first part of Theorem~\ref{The:maintheorem} says that this is indeed the case.

A related problem from \cite{planarPB} asks whether there is a planar, hyperbolic graph with bounded degrees, no unbounded faces, and the Liouville property. Combined with a result of \cite{UKtrans} showing that the Liouville property implies amenability in this context, the third part of Theorem~\ref{The:maintheorem} implies that such a graph would need to have accumulation points or satisfy no isoperimetric inequality. (Note that such a graph could have bounded codegree.)

\medskip
The aforementioned example from \cite{planarPB} shows that non-amenability implies neither hyperbolicity nor bounded codegree, and is one of the examples needed to show that no one of the four properties implies any of the other in $\pp$ (with the exception of non-amenability implying weak non-amenability). We now describe other examples showing the independence of those properties.

To prove that bounded codegree does not imply hyperbolicity or that weak non-amenability does not imply non-amenability it suffices to consider the square grid $\Z^2$.


To prove that hyperbolicity does not imply weak non-amenability nor bounded codegree, we adopt an example suggested by B.~Bowditch (personal communication). Start with a hyperbolic graph $G$ of bounded codegree $\Delta(G^*)$ and perform the following construction on any infinite sequence $\{F_n\}$ of faces of $G$. Enumerate the vertices of $F_n$ as $f_1, \ldots f_k$ in the order they appear along $F_n$ starting with an arbitrary vertex.
Add a new vertex $v_n$ inside $F_n$, and join it to each $f_i$ by a path $P_i$ of length $n$, so that the $P_i$ meet only at $v_n$. 
Then for every $1\leq i < k$, and every $j<n$, join the $j$th vertices of $P_i$ and $P_{i+1}$ with an edge. 
Call $G_1$ the resulting graph. Then $G_1$ has unbounded codegree, because $P_1, P_k$ and one of the edges of $F_n$ bound a face of length $2n-1$. Moreover $G_1$ is not weakly non-amenable: the set of vertices inside $F_n$ is unbounded in $n$, while its boundary has $|F_n|\leq\Delta(G^*)$ vertices. Finally, $G_1$ is hyperbolic: it is quasi-isometric to the graph obtained from $G$ by attaching a path $R$ of length $n$ to each $F_n$. 

To prove that bounded codegree does not imply weak non-amenability, consider the graph $G_2$ obtained from the same construction as above except that we now also introduce   edges between $P_k$ and $P_1$: now $G_2$ has bounded codegree while still not being weakly non-amenable.\footnote{B.~Bowditch (personal communication) noticed that $G_1$ is quasi-isometric to $G_2$, showing that having bounded codegree is not a quasi-isometric invariant in \pp, although he proved that having bounded codegree is a quasi-isometric invariant among weakly non-amenable graphs.}

To prove that hyperbolicity and weak non-amenability together do not imply non-amenability without the condition of no unbounded face consider the following example. Let $H$ be the hyperbolic graph constructed as follows. The vertex set of $H$ is the subset of $\R^2$ given by $\{(\frac{i}{2^n},n)\mid i\in\Z,n\in\N\}$. Join two vertices $(\frac{i}{2^n},n),(\frac{j}{2^m},m)$ with an edge whenever either $n=m$ and $i=j+1$ or $n=m+1$ and $i=2j$. The finite graph $H(a)$ is the subgraph of $H$ induced by those vertices contained inside the square with corners $(0,0),(a,0),(a,0),(a,a)$. We construct the graph $G$ by attaching certain $H(n)$ to $H$ as follows. For every $n\in \N$, attach a copy of $H(n)$ along the path $\{(n^2,0),\ldots,(n^2+n,0)\}$ of $H$ by identifying the vertex $(n^2+k,0)$ of $H$ with the vertex $(k,0)$ of $H(n)$. Note that the resulting graph $G$ is planar because $n^2+n< (n+1)^2$, and so the $H(n)$ we attached to $H$ do not overlap. It is easy to prove that $G$ is amenable and weakly non-amenable. It is also not hard to see that $G$ is hyperbolic, for example by noticing that the ray $\{(x,0),x\in\Z\}\subset G$ contains the only geodesic between any two of its vertices, and recalling that a graph  $G$ is hyperbolic, if any two geodesics of $G$ are either at bounded distance or diverge exponentially \cite{hypgroupnotes}.

To see that Theorem~\ref{The:maintheorem} becomes false if we allow accumulation points of vertices, consider the free product of the square grid $\Z^2$ with the line $\Z$; this graph can be embedded with bounded codegree, and it is non-amenable but not hyperbolic.

\section{Definitions} \label{defs}
The \textit{degree} $\deg(v)$ of a vertex $v$ in a graph $G$ is the number of edges incident with $v$; if
\[\Delta(G):=\sup_{v\in V(G)}\deg(v)\]
is finite we will say that $G$ has \textit{bounded degree}.
 
An \textit{embedding} of a graph $G$ in the plane will always mean a topological embedding of the corresponding 1-complex in the euclidean plane $\R^2$; in simpler words, an embedding is a drawing in the plane with no two edges crossing. A \textit{plane graph} is a graph endowed with a fixed embedding. A plane graph is \textit{accumulation-free} if its set of vertices has no accumulation point in the plane.

A \textit{face} of an embedding $\sigma: G \rightarrow \R^2$ is a component of $\R^2 \setminus \sigma(G)$. The \textit{boundary} of a face $F$ is the set of vertices and edges of $G$ that are mapped by $\sigma$ to the closure of $F$. The \textit{length} $|F|$ of $F$ is the number of edges in its boundary. A face $F$ is \textit{bounded} if the length $|F|$ is finite. If
\[\Delta(G^*):=\sup_{F\text{ bounded face of } G}|F|\]
is finite we will say that $G$ has \textit{bounded codegree}.

The \textit{Cheeger constant} of a graph $G$ is
\[c(G):=\inf_{\emptyset\neq S\subset G\text{ finite}}\dfrac{|\partial S|}{|S|},\]
where $\partial S=\{v\in G\setminus S\mid\text{there exists }w\in S\text{ adjacent to }v\}$ is the \textit{boundary} of $S$. Graphs with strictly positive Cheeger constant are called \textit{non-amenable} graphs. A graph is \textit{weakly non-amenable} if satisfies an isoperimetric inequality of the form $|S|\leq f(|\partial S|)$ for all non-empty finite vertex sets $S$, where $f\colon\N\rightarrow\N$ is a monotone increasing diverging function.

A $x$-$y$ path in a graph $G$ is called a \textit{geodesic} if its length coincides with the distance between $x$ and $y$. A \textit{geodetic triangle} consists of three vertices $x,y,z$ and three geodesics, called its \textit{sides}, joining them. A geodetic triangle is $\delta$-\textit{thin} if each of its sides is contained in the $\delta$-neighbourhood of the union of the other two sides. A connected graph is $\delta$-\textit{hyperbolic} if each geodetic triangle is $\delta$-thin. The smallest such $\delta\geq0$ will be called the \textit{hyperbolicity constant} of $G$. A graph is \textit{hyperbolic} if there exists a $\delta\geq0$ such that each connected component is $\delta$-hyperbolic.

If $C$ is a cycle of $G$ and $x,y$ lie on $C$ then they identify two arcs joining them along $C$: we will write $xCy$ and $yCx$ for these paths. Similarly, if $P$ is a path passing through these vertices, $xPy$ is the sub-path of $P$ joining them.

\section{Hyperbolicity and weak non-amenability imply bounded codegree}
In this and the following sections we will prove each of the three implications of Theorem~\ref{The:maintheorem} separately.

We will assume throughout the text that $G\in \pp$, i.e. $G$ is a accumulation-free plane graph with bounded degrees, fixed for the rest of this paper. Theorem~\ref{The:maintheorem} is trivial in the case of forests, so from now on we will assume that $G$ has at least a cycle, or in other words it has a bounded face.

A \textit{geodetic cycle} $C$ in a graph $G$ is a cycle with the property that for every two points $x,y\in C$ at least one of $xCy$ and $yCx$ (defined in the end of Section~\ref{defs}) is a geodesic in $G$.

\begin{lemma}
If $G\in \pp$ is hyperbolic, then the lengths of its geodetic cycles are bounded, i.e.
\[\sup_{C\text{ geodetic cycle of }G}|C|<\infty.\]
\label{Lem:boundedgeocycle}
\end{lemma}
\begin{proof}
Let $\delta$ be the hyperbolicity constant of $G$. We will show that no geodetic cycle has more than $6\delta$ vertices.

Let $C$ be a geodetic cycle, say with $n$ vertices, and choose three points $a,b,c$ on $C$ as equally spaced as possible, i.e. every pair is at least $\Big\lfloor\dfrac{n}{3}\Big\rfloor$ apart along $C$. Let $ab$ be the arc of $C$ joining $a$ and $b$ that does not contain $c$, and define $bc$ and $ca$ similarly. We want to show that $ab,bc$ and $ca$ form a geodetic triangle.

If $x,y,z$ are distinct points in $C$ then let $xzy$ be the arc in $C$ from $x$ to $y$ that passes through $z$. Then we know that one of $ab,acb$ is a geodesic joining $a$ and $b$, and $|acb|\geq2\Big\lfloor\dfrac{n}{3}\Big\rfloor>|ab|$, so $ab$ is a geodetic arc. Similarly, $bc$ and $ca$ are geodetic arcs. 

Consider now the point $p$ on $ab$ at distance $\Big\lfloor\dfrac{n}{6}\Big\rfloor$ from $a$ along $C$. Since $G$ is a $\delta$-hyperbolic graph, we know that there is a vertex $q$ on $bc$ or $ca$ which is at distance at most $\delta$ from $p$. But as $C$ is a geodetic cycle, the choice of $a,b,c$ implies that
\[d(p,q)\geq\min\{d(p,a),d(p,b)\}=\Big\lfloor\dfrac{n}{6}\Big\rfloor,\]
from which we deduce that $n\leq 6\delta$.
\end{proof}

By the Jordan curve theorem, we can say that a point of $\R^2$ is \textit{strictly inside} a given cycle $C$ of $G$ if it belongs to the bounded component of $\R^2\setminus C$ and is \textit{inside} $C$ if it belongs to $C$ or is strictly inside. We say that a subset of $\R^2$ is inside (resp. strictly inside) $C$ if all of its points are inside (resp. strictly inside) $C$. A subset of $\R^2$ is \textit{outside} $C$ if it is not strictly inside $C$.

Recall that we are assuming $G$ to have no accumulation point, so inside each cycle we can only have finitely many vertices.

\begin{corollary}
\label{Cor:geocyclesuffices}
Suppose $G\in \pp$ is hyperbolic and weakly non-amenable. If every face of $G$ is contained inside a geodesic cycle, then $\Delta(G^*)<\infty$.
\end{corollary}
\begin{proof}
Consider a face $F$ contained inside a geodetic cycle $C$; by Lemma~\ref{Lem:boundedgeocycle} we know that $|C|\leq 6\delta$, where $\delta$ is the hyperbolicity constant of $G$. Let $S$ be the set of all vertices inside the geodetic cycle $C$ so $|S|<\infty$ as there is no accumulation point. Then the vertices of $S$ sending edges to the boundary $\partial S$ belong to $C$ and each vertex of $C$ sends less than $\Delta(G)$ edges to $\partial S$, implying that $|\partial S|<\Delta(G)|C|$. Let $f\colon\N\rightarrow\N$ be a monotone increasing diverging function witnessing the weak non-amenability of $G$. Then, since $F\subseteq S$,
\[|F|\leq|S|\leq f(|\partial S|)\leq f(6\delta\Delta(G)),\]
which is uniformly bounded for every face $F$ of $G$.
\end{proof}
In what follows we will exhibit a construction showing that in any graph each face is contained inside a geodetic cycle, which allows us to apply Corollary~\ref{Cor:geocyclesuffices} whenever the graph is hyperbolic and weakly non-amenable.

We remarked above that by the Jordan curve theorem we can make sense of the notion of being contained inside a cycle. Similarly, given three paths $A,B,C$ sharing the same endpoints, if $A\cup C$ is a cycle and $B$ lies inside it, we will say that $B$ is \textit{between} $A,C$.

Now, given a cycle $C$ and two points $x,y\in C$, consider the set $\mathcal{S}=\mathcal{S}(x,y)$ of $x$-$y$ geodesics that lie outside $C$. This set can be divided into two classes:
\begin{gather*}
\mathcal{S}_1:=\{\Gamma\in\mathcal{S}\mid xCy\text{ is between }yCx,\Gamma\},\\
\mathcal{S}_2:=\{\Gamma\in\mathcal{S}\mid yCx\text{ is between }xCy,\Gamma\}.
\end{gather*}

For the proof of Theorem~\ref{The:maintheorem}, we will make use of the notion of `the closest geodesic' to a given cycle; let us make this more precise. Consider a cycle $C$ in a plane graph, two points $x$ and $y$ on $C$ and a choice of an arc on $C$ joining them, say $xCy$.
Let us define a partial order on the set $\mathcal{S}_1$ defined above: for any two geodesics $\Gamma,\Gamma'\in\mathcal{S}_1$ we declare $\Gamma\preceq \Gamma'$ if $\Gamma$ is between $xCy,\Gamma'$.

\begin{lemma}
With notation as above, $(\mathcal{S}_1,\preceq)$ has a least element.
\label{Lem:Shasminimum}
\end{lemma}
\begin{proof}
The set $\mathcal{S}_1$ is a subset of all paths from $x$ to $y$ of length $d(x,y)$. These paths are contained in the ball of center $x$ and radius $d(x,y)$. As $G$ is locally finite, this ball is finite and so is $\mathcal{S}_1$. Therefore, it suffices to produce for every couple of elements a (greatest) lower bound.\footnote{In a very similar way we can produce a least upper bound, showing that $\mathcal{S}_1$ is a finite lattice.}

Pick two geodesics $\Gamma,\Gamma'$ in $\mathcal{S}_1$; let $P_1,\ldots,P_h$ be the collection (ordered from $x$ to $y$) of maximal subpaths of $\Gamma$ lying inside the cycle $xCy\cup\Gamma'$ and $Q_1,\ldots,Q_k$ the collection (ordered from $x$ to $y$) of maximal subpaths of $\Gamma'$ lying inside the cycle $xCy\cup\Gamma$ (note that $h-k\in\{-1,0,1\}$). Without loss of generality, we can assume that $x$ belongs to $P_1$, so $h-k\in\{0,1\}$.

Now consider the subgraph
\[\Gamma'':=\begin{cases}
P_1\cup Q_1\ldots\cup P_h\cup Q_k, & \text{if }h=k;\\
P_1\cup Q_1\ldots\cup P_{h-1}\cup Q_k\cup P_h, & \text{if }h=k+1.
\end{cases}\]
Note that each $P_i$ shares one endvertex with $Q_{i-1}$ and the other with $Q_i$, and the same holds for each $Q_j$. We want to prove $\Gamma''$ to be an element of $\mathcal{S}_1$ and specifically the greatest lower bound of $\Gamma$ and $\Gamma'$.\footnote{The differences between the two cases in the definition of $\Gamma''$ are actually irrelevant.}

Note that $\Gamma$ and $\Gamma'$ intersect in some points $x=x_1,x_2,\ldots,x_n=y$ (the endvertices of all $P_i$ and $Q_j$) and, being geodesics, $x_i\Gamma x_{i+1}$ is as long as $x_i\Gamma'x_{i+1}$. This implies $|\Gamma''|=|\Gamma'|=|\Gamma|=d(x,y)$, i.e. $\Gamma''$ is a geodesic (in particular, it is a path). The fact that $\Gamma''\in\mathcal{S}_1$ follows from having put together only sub-paths of elements from $\mathcal{S}_1$. Lastly, we need to show that both $\Gamma''\preceq \Gamma$ and $\Gamma''\preceq \Gamma'$ hold. But all paths $P_i$ and $Q_j$ are inside both $xCy\cup \Gamma$ and $xCy\cup \Gamma'$, therefore so is $\Gamma''$.
\end{proof}

Let us say that in a plane graph a path $P$ \textit{crosses} a cycle $C$ if the endpoints of $P$ are outside $C$ but there is at least one edge $e$ of $P$ such that the interior of the curve in $\R^2$ representing $e$ lies strictly inside $C$.

\begin{corollary}
Consider a cycle $C$ in a plane graph $G$ and two points $x$ and $y$ on $C$. Then there exists a $x$-$y$ geodesic $\Gamma$ of $G$ satisfying the following:
\begin{itemize}
\item[(1)] $xCy$ is between $yCx,\Gamma$;
\item[(2)] there is no geodesic outside $C$ crossing the cycle $xCy\cup\Gamma$.
\end{itemize}
\label{Cor:closestgeodesic}
\end{corollary}
\begin{proof}
Note that the condition (1) is exactly the definition of the set $\mathcal{S}_1$ given above so by Lemma~\ref{Lem:Shasminimum}, there exists a least $x$-$y$ geodesic $\Gamma$ with respect to $\preceq$. Let us show that this is the required geodesic. Suppose there is a geodesic $\Gamma'$ that crosses the cycle $xCy\cup\Gamma$ and lies outside $C$, so the endpoints of $\Gamma'$ are outside $yCx\cup\Gamma$ and $\Gamma'$ has an edge $e$ strictly inside $xCy\cup \Gamma$. Let $a\Gamma'b$ the longest subpath of $\Gamma'$ containing $e$ and lying inside $xCy\cup\Gamma$. Then $a,b$ are on $C$ and the geodesic
\[\Gamma'':=x\Gamma a\cup a\Gamma'b\cup b\Gamma y\]
satisfies $\Gamma''\prec\Gamma$, contradicting the minimality of $\Gamma$. This contradiction proves our claim.
\end{proof}
Note that the Corollary does not claim uniqueness for the geodesic: if $\Gamma$ satisfies the claim and $\Gamma'\preceq\Gamma$ then $\Gamma'$ satisfies it as well. However, in the proof we showed that the unique least element of $\mathcal{S}_1$ satisfies the claim: such a geodesic will be referred to as the \textit{closest} geodesic to the cycle $C$ in $\mathcal{S}_1$. We conclude that, given a pair of points $x,y$ on a cycle $C$, there are exactly two $x$-$y$ geodesics closest to $C$: one for each of $\mathcal{S}_1,\mathcal{S}_2$. These two geodesics can intersect, but cannot cross each other.

\begin{theorem}
\label{The:boundedcodegree}
If $G\in \pp$ is hyperbolic and weakly non-amenable then $\Delta(G^*)<\infty$.
\end{theorem}
\begin{proof}
We want to show that if $F$ is a face of $G$, then it is contained in a geodetic cycle and then apply Corollary~\ref{Cor:geocyclesuffices}. The idea of the proof is to construct a sequence of cycles $C_0,C_1,\ldots$ each containing $F$, with the lengths $|C_i|$ strictly decreasing, so that the sequence is finite and the last cycle is a geodetic cycle.

Let us start with the cycle $C_0$ coinciding with the boundary of the face $F$. If $C_0$ is geodetic we are done, otherwise there are two points $x,y$ such that both $xC_0y$ and $yC_0x$ are not geodesics. Consider a geodesic $\Gamma_1$ joining them: since $F$ is a face, $\Gamma_1$ must lie outside the cycle $C_0$. Therefore, we have three paths $xC_0y$, $yC_0x$ and $\Gamma_1$ between $x$ and $y$. Assume without loss of generality that $xC_0y$ is between $yC_0x,\Gamma_1$. Then the union of $\Gamma_1$ with $yC_0x$ yields a new cycle $C_1$ with the following properties:
\begin{itemize}
\item $|C_1|=|yC_0x|+|\Gamma_1|<|yC_0x|+|xC_0y|=|C_0|$, since $xC_0y$ is not a geodesic while $\Gamma_1$ is;
\item the face $F$ is inside the cycle $C_1$ since it was inside (or rather, equal to) $C_0$ which in turn is inside $C_1$.
\end{itemize}
Using Lemma~\ref{Lem:Shasminimum} we can require the geodesic $\Gamma_1$ to be the closest to the cycle $C_0$ with respect to the arc $xC_0y$. Note that the cycle $C_1$ cannot be crossed by any geodesic: a side of the cycle is made by a face, which does not contain any strictly inner edge, and the other side is bounded by the closest geodesic, which cannot be crossed by Corollary~\ref{Cor:closestgeodesic}.

We can iterate this procedure: assume by induction that after $n$ steps, we are left with a cycle $C_n$ such that the face $F$ is still inside $C_n$ and $C_n$ cannot be crossed by geodesics. If $C_n$ is a geodetic cycle we are done, otherwise there are two points $x,y\in C_n$ that prevent that, and we can find a closest geodesic $\Gamma_{n+1}$ as before, creating a new cycle $C_{n+1}$. We conclude that the face $F$ is inside $C_{n+1}$ and $|C_{n+1}|<|C_n|$. Since these lengths are strictly decreasing, the process halts after finitely many steps, yielding the desired geodetic cycle.
\end{proof}

\section{Non-amenability and bounded codegree imply hyperbolicity}
One of the assertions of Theorem~\ref{The:maintheorem} was proved in \cite{NorCir} using random walks. In this section we provide a purely geometric proof of that statement.

Bowditch proved in \cite{BowNot} many equivalent conditions for hyperbolicity of metric spaces, one of which is known as \textit{linear isoperimetric inequality}. For our interests, which are planar graphs of bounded degree, 
that condition has been rephrased as in Theorem~\ref{The:Bowditch'scondition} below. Before stating it we need some definitions.

Let us call a finite, connected, plane graph $H$ with $\delta(H)\geq2$ a \textit{combinatorial disk}. Note that all faces of $H$ are bounded by a cycle; let us call $\partial_{top}H$ the cycle bounding the unbounded face of $H$.

\begin{definition}
A combinatorial disk $H$ satisfies a $(k,D)$\textit{-linear isoperimetric inequality} (LII) if $|F|\leq D$ for all bounded faces $F$ of $H$ and the number of bounded faces of $H$ is bounded above by $k|\partial_{top}H|$.
\end{definition}

\begin{definition}
\label{Def:LIIforinfinitegraphs}
An infinite, connected, plane graph $G$ satisfies a LII if there exist $k,D\in \N$ such that the following holds: for every cycle $C\subset G$ there is a combinatorial disk $H$ satisfying a $(k,D)$-LII 
and a map $\varphi\colon H\rightarrow G$ which is a graph-theoretic isomorphism onto its image (so that $\varphi$ does not have to respect the embeddings of $H,G$ into the plane), such that $\varphi(\partial_{top}H)=C$.
\end{definition}
Bowditch's criterion is the following: 
\begin{theorem}[\cite{BowNot}]
\label{The:Bowditch'scondition}
A plane graph $G$ of minimum degree at least 3 and bounded degree is hyperbolic if and only if  $G$ satisfies a LII. 
\end{theorem}

An immediate corollary is the following:
\begin{corollary}
\label{Cor:LIIinsidecycles}
Let $G$ be a plane graph of minimum degree at least 3, bounded degree and codegree. Suppose there exists $k$ such that for all cycles $C\subset G$ the number of faces of $G$ inside $C$ is bounded above by $k|C|$. Then $G$ is hyperbolic.
\end{corollary}
\begin{proof}
For every cycle $C$, let $H$ be the subgraph of $G$ induced by all vertices inside $C$. Then $H$ is a finite plane graph of codegree bounded above by $\Delta(G^*)$. By assumption, the number of bounded faces of $H$ is bounded above by $k|C|=k|\partial_{top}H|$. Thus $G$ satisfies a LII, and $G$ is hyperbolic by Theorem~\ref{The:Bowditch'scondition}. 
\end{proof}
We will see a partial converse of this statement in Lemma~\ref{Lem:converseofLII}. \\
We would like to apply this criterion to our non-amenable, bounded codegree graph $G$, but $G$ might have minimum degree less than 3. Therefore, we will 
perform on $G$ the following construction in order to obtain a graph $G'$ of minimum degree 3 without affecting any of the other properties of $G$ we are interested in.

Define a \textit{decoration} of $G$ to be a maximal connected induced subgraph $H$ with at most 2 vertices in the boundary $\partial H$. Perform the following procedure on each decoration $H$ of the graph $G$: if $|\partial H|=1$ delete $H$, while if $\partial H=\{v,w\}$ delete $H$ and add the edge $\{v,w\}$ if not already there. Call the resulting graph $G'$. Note that the minimum degree of $G'$ is at least 3: any vertex of $G$ of degree at most 2 belongs to a decoration, and if $H$ is a decoration and $x\in\partial H$ then by maximality $x$ sends at least 3 edges to $G\setminus(H\cup\partial H)$ when $|\partial H|=1$ and at least 2 edges when $|\partial H|=2$. Note also that the maximum degree of $G'$ is at most $\Delta(G)$. Note also that the maximum degree of $G'$ is at most $\Delta(G)$.

Now assume $G$ is non-amenable with Cheeger constant $c(G)$; then the size of decorations is bounded above by $\frac{2}{c(G)}$ and thus the size of any face of $G$ is reduced by at most $\frac{2}{c(G)}$ after the procedure, so $\Delta(G'^*)$ is finite if $\Delta(G^*)$ is. Consider the identity map $I\colon V(G')\hookrightarrow V(G)$. Then
\[d_{G'}(x,y)\leq d_G(I(x),I(y))\leq \frac{2}{c(G)}d_{G'}(x,y)\]
and every vertex in $G$ is within $\frac{2}{c(G)}$ from a vertex of $f(V(G'))$, hence $I$ is a quasi-isometry between $G$ and $G'$. Thus $G$ is non-amenable, since non-amenability is a quasi-isometric invariant for graphs of bounded degree (see for instance Theorem~11.10 in \cite{DruKapLec}). For the same reason, if we can prove that $G'$ is hyperbolic then so is $G$.
\begin{theorem}
\label{The:hyperbolic}
If $G\in \pp$ is non-amenable and it has bounded codegree then $G$ is hyperbolic.
\end{theorem}
\begin{proof}
Starting from $G$, perform the construction of the auxiliary graph $G'$ as above: the resulting graph $G'$ is non-amenable, has bounded codegree and has minimum degree at least 3.

Let $C$ be a cycle and $S\subset G'$ the (finite but possibly empty) subset of vertices lying strictly inside $C$; by non amenability we have
\[|C|\geq|\partial S|\geq c(G')|S|.\]
Let us focus on the finite planar graph $G'[C\cup S]$ induced by $C\cup S$ and let $F$ be the number of faces inside it. Since each vertex is incident with at most $\Delta(G)$ faces, we have $|C\cup S|\Delta(G)\geq F$. Thus
\[(1+c(G'))|C|\geq c(G')(|S|+|C|)\geq c(G')\frac{1}{\Delta(G)}F\]which is equivalent to $F\leq\frac{(1+c(G'))\Delta(G)}{c(G')}|C|$. Since $\Delta(G'^*)$
is finite, by Corollary~\ref{Cor:LIIinsidecycles} $G'$ is hyperbolic. By the remark above, $G$ is hyperbolic too.
\end{proof}

\section{Hyperbolicity and weak non-amenability imply non-amenability}
Let us prepare the last step of the proof of Theorem~\ref{The:maintheorem} with a Lemma.
\begin{lemma}
\label{Lem:closedwalk}
Suppose $G$ has bounded codegree and no unbounded faces. Then for every finite connected induced subgraph $S$ of $G$, there exists a closed walk $C$ such that $S$ is inside $C$ and at least $|C|/\Delta(G^*)$ vertices of $C$ are in the boundary of $S$.
\end{lemma}
\begin{proof}
Let $H$ be the subgraph of $G$ spanned by $S$ and all its incident edges. Note that $H$ contains all vertices in $\partial S$, but no edges joining two vertices of $\partial S$. Then $H$ is a finite plane graph by definition. We let $C$ be the closed walk bounding the unbounded face of $H$. We claim that $C$ has the desired properties.

To see this, let $x_1,\ldots,x_n$ be an enumeration of the vertices of $\partial S$ in the order they are visited by $C$. Then the subwalk $x_iCx_{i+1}$ is contained in some face of $G$: all interior vertices of $x_iCx_{i+1}$ lie in $S$ by our definitions,and so all edges incident with those vertices are in $H$; therefore, since $x_iCx_{i+1}$ is a facial walk in $H$, it is also a facial walk in $G$. Since $x_iCx_{i+1}$ is contained in some face of $G$ and $G$ has only bounded faces, we have $|x_iCx_{i+1}| < \Delta(G^*)$. Applying this to all $i$, we obtain 

\[|C\cap \partial S|=n>\sum_{i=1}^n\dfrac{|x_iCx_{i+1}|}{\Delta(G^*)}\geq\dfrac{|C|}{\Delta(G^*)}.\qedhere\]
\end{proof}

We need a result which is almost a converse of Corollary~\ref{Cor:LIIinsidecycles}.
\begin{lemma}
\label{Lem:converseofLII}
Let $G\in \pp$ be hyperbolic and weakly non-amenable. Then there exists $k$ such that for all cycles $C\subset G$ the number of faces of $G$ inside $C$ is bounded above by $k|C|$. 
\end{lemma}
\begin{proof}
Using the auxiliary graph $G'$ from the previous section, we may assume that $G$ has minimum degree at least 3. Let $C$ be a cycle of $G$. Since $G$ is hyperbolic, by Theorem~\ref{The:Bowditch'scondition} there exists a combinatorial disk $H$ satisfying a $(k',D)$-LII with an isomorphism $\varphi$ from $H$ to a subgraph of $G$ such that $\varphi(\partial_{top}H)=C$. The boundaries $F_i,i\in I,$ of bounded faces of $H$ are sent by $\varphi$ to cycles $C_i:=\varphi(F_i)$ of $G$ so that $|I|\leq k'|C|$, and the bound $D$ on the length of bounded faces of $H$ is an upper bound to the length of those cycles $C_i$. Let $S_i$ be the (finite) set of vertices of $G$ strictly inside $C_i$, so that $\partial S_i\subseteq C_i$. Let $f\colon\N\rightarrow\N$ a monotone increasing diverging function witnessing the weak non-amenability of $G$, i.e. $|S|\leq f(|\partial S|)$ for all finite non-empty $S\subset G$. Then
\[|S_i|\leq f(|\partial S_i|)\leq f(|C_i|)\leq f(D),\]
for all nonempty $S_i$. Let $F(C_i)$ be the number of faces of $G$ inside $C_i$; then for a nonempty $S_i$ we have $F(C_i)\leq\Delta(G)|S_i|$ because no vertex can meet more than $\Delta(G)$ faces. In conclusion
\[|\{\text{faces inside }C\}|=\sum_{i\in I} F(C_i)\leq\sum_{i\colon S_i=\emptyset}1+\sum_{i\colon S_i\neq\emptyset}\Delta(G)|S_i|\leq|I|+\Delta(G)f(D)|I|,\]
from which by setting $k:=k'+\Delta(G)f(D)k'$ the assertion follows. 
\end{proof}

Note that in order to prove the non-amenability of a graph $G$ it suffices to check that $|\partial S|\geq c|S|$ for some constant $c>0$ and all finite induced connected subgraphs $S$, instead of all finite subsets. Indeed, if we assume so and if $S$ is a finite induced subgraph with components $S_1,\ldots,S_n$, then
\[|\partial S|=\left|\bigcup_{i=1}^n\partial S_i\right|\geq\dfrac{1}{\Delta(G)}\sum_{i=1}^n |\partial S_i|\geq \dfrac{c}{\Delta(G)}\sum_{i=1}^n|S_i|=\dfrac{c}{\Delta(G)}|S|,\]
where the first equality follows from $\partial S_i\cap S_j=\emptyset$ for all $i\neq j$ ($S$ is induced) and the first inequality holds because the boundaries $\partial S_i$ can overlap, but no vertex of $\partial S$ belongs to more than $\Delta(G)$ of them.

\begin{theorem}
\label{The:nonamenable}
If $G\in \pp$ is hyperbolic and weakly non-amenable then $G$ is non-amenable. 
\end{theorem}
\begin{proof}
By the above consideration it is enough to check the non-amenability only on connected induced subgraphs of $G$. By Theorem~\ref{The:boundedcodegree} we know that $G$ has bounded codegree. Let $S$ be such a subgraph and $C$ as in Lemma~\ref{Lem:closedwalk}. Then
\[|\partial S|\geq|\partial S\cap C|\geq\frac{|C|}{\Delta(G^*)}\]
and thus $\dfrac{|\partial S|}{|S|}\geq\dfrac{1}{\Delta(G^*)}\dfrac{|C|}{|S|}$.

Let $k>0$ be as in Lemma~\ref{Lem:converseofLII}; if $T$ denotes the set of all vertices inside $C$ and $F$ the set of all faces inside $C$, we have
\[\dfrac{|C|}{|T|}=\dfrac{|C|}{|F|}\cdot\dfrac{|F|}{|T|}\geq\dfrac{1}{k}\dfrac{1}{\Delta(G^*)},\]
since each face is incident with at most $\Delta(G^*)$ vertices.
Combining the last two inequalities, we have
\[\dfrac{|\partial S|}{|S|}\geq \dfrac{1}{\Delta(G^*)}\dfrac{|C|}{|S|}\geq\dfrac{1}{\Delta(G^*)}\dfrac{|C|}{|T|}\geq \dfrac{1}{k(\Delta(G^*))^2}.\qedhere\]
\end{proof}

\section{Graphs with unbounded degrees}

We provided enough examples to show that Theorem~\ref{The:maintheorem} is best possible, except that we do not yet know to what extent the bounded degree condition is necessary. Solutions to the following problems would clarify this. Let now $\pp^*$ denote the class of plane graphs with no accumulation point of vertices; so that $\pp$ is the subclass of bounded degree graphs in $\pp^*$.

\begin{problem}
Is there a hyperbolic, amenable, weakly non-amenable plane graph of bounded codegree and no unbounded face in $\pp^*$?
\end{problem}

\begin{problem}
Is every non-amenable bounded codegree graph in $\pp^*$ hyperbolic?
\end{problem}

\clearpage{\pagestyle{empty}\cleardoublepage}
\pagestyle{plain}
\bibliographystyle{abbrv}
\bibliography{biblio}

\end{document}